\documentclass[11pt]{article}
\usepackage{amsmath,amsthm,verbatim,amsfonts,amscd, graphicx, yhmath}
\usepackage{hyperref}
\usepackage{parskip}
\usepackage{xparse}
\usepackage{relsize}

\usepackage[utf8]{inputenc}    
\usepackage[T1]{fontenc} 
\usepackage{tikz-cd}
\usepackage{mathtools} 
\usepackage{mathrsfs}
\usepackage{enumerate} 
\usepackage{braket}
\usepackage{bigints}
\usepackage{amssymb}
\let\amslrcorner\lrcorner
\usepackage{MnSymbol}
\let\lrcorner\amslrcorner

\theoremstyle{plain}
\newtheorem{theorem}{Theorem}[section]

\newtheorem{corollary}{Corollary}[section]

\newtheorem*{claim*}{Claim}
\newtheorem*{lemma*}{Lemma}
\newtheorem*{theorem*}{Theorem}

\newtheorem{remark}{Remark}

\theoremstyle{definition}

\DeclareMathOperator{\Hessian}{Hess}
\DeclareMathOperator{\distance}{dist}

\DeclareMathOperator{\D}{\mathrm{D}}
\DeclareMathOperator{\N}{\mathrm{N}}

\renewcommand{\Re}{\operatorname{Re}}

\makeatletter
\newcommand{\extp}{\@ifnextchar^\@extp{\@extp^{\,}}}
\def\@extp^#1{\mathop{\bigwedge\nolimits^{\!#1}}}
\makeatother

\allowdisplaybreaks
\everymath{\displaystyle}

\NewDocumentCommand{\inn}{mo}{%
	\langle #1\rangle
	\IfValueT{#2}{^{}_{\mspace{-3mu}#2}}%
}

\NewDocumentCommand{\dinn}{mo}{%
	\llangle #1\rrangle
	\IfValueT{#2}{^{}_{\mspace{-3mu}#2}}%
}

\NewDocumentCommand{\xinn}{>{\SplitArgument{1}{,}}mo}{%
	\doinnerproduct#1
	\IfValueT{#2}{^{}_{\mspace{-3mu}#2}}%
}
\NewDocumentCommand{\doinnerproduct}{mm}{%
	\langle #1\mid #2\rangle 
}

\makeatletter
\newcommand*\bigcdot{\mathpalette\bigcdot@{1}}
\newcommand*\bigcdot@[2]{\mathbin{\vcenter{\hbox{\scalebox{#2}{$\m@th#1\bullet$}}}}}
\makeatother

\usepackage[
    backend=biber, 
    natbib=true,
    url=false,
]{biblatex}

\IfFileExists{bibliography.bib}
  {\addbibresource{bibliography.bib}}
  {}

\IfFileExists{../bibliography.bib}
  {\addbibresource{../bibliography.bib}}
  {}

\begin{document}
	
	\title{Pseudoconvexity of Level Sets}
	
	\author{Bingyuan Liu}

	\date{\today}

	\maketitle

  \begin{abstract}
    We investigate the pseudoconvexity of level sets of the distance function. First, we show the formula of \cite{Li17} for the Diederich--Forn\ae ss index can be extended to characterize pseudoconvexity. This extension implies that the rate of change of the Levi form is nonincreasing at weakly pseudoconvex points. Additionally, we demonstrate that perturbing the distance function by a factor of $e^{-t|z|^2}$ ensures all level sets sufficiently close to the boundary are strongly pseudoconvex for all $t>0$.
\end{abstract}

\section{Introduction}

Let $\Omega$ be a bounded domain in $\mathbb{C}^n$ with a smooth boundary. We say $\Omega$ is Hartogs pseudoconvex if $-\log(-\delta)$ is plurisubharmonic in $\Omega$, where $\delta(x)$ is the signed distance function:
\[
\delta(x) =
\begin{cases}
\operatorname{dist}(x, \partial\Omega) & \text{if } x \in \Omega^c \\
-\operatorname{dist}(x, \partial\Omega) & \text{if } x \in \Omega
\end{cases}
\]
Alternatively, $\Omega$ is called Levi pseudoconvex if the Levi form $\Hessian_\delta (L_p,L_p) \geq 0$ on $\partial\Omega$ for any $L_p \in T_p^{1,0}(\partial\Omega)$, where $p \in \partial\Omega$. The equivalence between Levi pseudoconvexity and Hartogs pseudoconvexity is a remarkable result discovered in the early 20th century. By the resolution of the Levi problem, these domains are exactly the domains of holomorphy; that is, for every boundary point, there exists a function in $\mathcal{O}(\Omega)$ which cannot be locally extended beyond that point.

Historically, the concept of pseudoconvexity was developed as the complex analogue to real convexity. For a convex domain, at any boundary point, there exists a supporting real hyperplane that remains strictly outside the interior of the domain. Strongly pseudoconvex domains exhibit a similar geometric property: at any boundary point, there exists a local complex analytic variety separated from the interior of the domain. However, for weakly pseudoconvex domains---where the Levi form is only positive semi-definite---this geometric analogy breaks down. A local supporting complex analytic variety is not guaranteed to exist, a phenomenon famously demonstrated by \citet{KN73}. Indeed, many properties are different between pseudoconvexity and convexity. Furthermore, the topological behavior of real convexity does not naively translate to complex pseudoconvexity. In a convex domain, if the boundary of a line segment (its two endpoints) lies entirely within the domain, the entire segment must also lie within it. However, for a pseudoconvex domain, the fact that the boundary of an analytic disc lies within the domain does not guarantee that the interior of the disc is contained within it. A simple counterexample is an annulus in $\mathbb{C}$, which is pseudoconvex. A circle can be placed safely between the inner and outer boundaries, but the interior of that bounding disc will inevitably cover the central hole, falling outside the domain.

Because the global geometry of pseudoconvex domains diverges so sharply from that of real convex domains, mathematicians have sought stricter refinements to recover stronger global properties. A fundamental refinement is the Diederich--Forn\ae ss index (see \citet{CF11} for the definition), introduced by \citet{DF77b}. Rather than relying on the standard plurisubharmonicity of $-\log(-\delta)$, this index quantifies pseudoconvexity by measuring the supremum of exponents $\eta \in (0, 1]$ for which there exists a defining function $\rho$ such that $-(-\rho)^\eta$ is strictly plurisubharmonic. This fractional exponent serves as a precise analytic gauge of the domain's boundary rigidity; notably, \citet{DF77a} demonstrated that it governs critical topological obstructions, such as the existence of a Stein neighborhood basis (see also \citet{FH08, Yu21}).



The Diederich--Forn\ae ss index is intrinsically linked to the derivative of the Levi form (see, e.g., \citet{Ad15}, \citet{FH08}, \citet{Ha17}). In \citet{Li17}, it was shown that the Diederich--Forn\ae ss index equals $\eta_0$ if and only if the following boundary condition holds:
\begin{align*}
    & \left(\frac{1}{1-\eta_\psi}-1\right)\left|(\bar{L} \psi)\left(N_r r\right)+\operatorname{Hess}_r\left(N_r, L\right)\right|^2+\frac{\|\nabla r\|}{2}\left(\frac{\|\nabla r\|}{2} \operatorname{Hess}_\psi(L, L)+\inn{\nabla_L \nabla_{N_r} \nabla r, L}\right)  \leq 0,
\end{align*} where, roughly speaking, $L$ is the tangential direction, $N$ is the normal direction, $r$ is the base defining function, $\psi$ is the weight function, and $\eta$ is related to the Diederich--Forn\ae ss index.

This condition, however, lacks a direct connection to pseudoconvexity. Setting $\eta_0=0$ reduces the inequality to $\left(\frac{\|\nabla r\|}{2} \operatorname{Hess}_\psi(L, L)+\inn{\nabla_L \nabla_{N_r} \nabla r, L}\right) \leq 0$, which further simplifies to $\inn{\nabla_L \nabla_{N_r} \nabla r, L}\leq 0$ when $\psi=0$. Yet, this specific reduction remains unproven, and the results in \citet{Li17} apply exclusively to the Diederich--Forn\ae ss index. In this paper, we establish a direct connection between pseudoconvexity and the third derivative $\inn{\nabla_L \nabla_{N_r} \nabla r, L}$. Furthermore, we leverage this relationship to characterize how strongly pseudoconvex domains approximate a weakly pseudoconvex domain from the interior. In particular, we prove the following two theorems.

\begin{theorem}\label{thm1}
    Let $\Omega$ be a bounded pseudoconvex domain in $\mathbb{C}^n$ with smooth boundary. Then \[\left.\inn{\nabla_L\nabla_{\N}\N, L}\right|_{U\cap\partial\Omega}\leq 0,\] on $\Sigma_L:=\left\{p \in U \cap \partial \Omega: \operatorname{Hess}_\delta(L, L)=0 \text { at } p\right\} $ for an arbitrary coordinate chart $U$.
\end{theorem}
\begin{remark}
    Throughout the remainder of the paper, we will generally suppress explicit mention of local coordinate charts. The primary purpose of restricting to a chart $U$ is to ensure the existence of a nonvanishing vector field $L \in \mathscr{C}^\infty(\overline{\Omega}, T^{(0,1)}\mathbb{C}^n)$ satisfying $L\delta = 0$, as such a field may not exist globally. Nevertheless, all subsequent calculations are intrinsically independent of the choice of coordinates.
\end{remark}
The notation for $L$ and $\N$ will be defined in Section \ref{pre}.

In particular, Theorem \ref{thm1} implies that the rate of change of the Levi form is nonpositive at the weakly pseudoconvex points (see Corollary \ref{cor1}). However,  it does not generally ensure strong pseudoconvexity. Indeed, if $\Omega$ is weakly pseudoconvex, $\Omega_{-\epsilon}$ may inherit this weak pseudoconvexity—a phenomenon observed in smoothed bidiscs where $\inn{\nabla_L\N, \N} = 0$. To guarantee that $\Omega_{-\epsilon}$ becomes strongly pseudoconvex immediately upon perturbing $\epsilon$ away from zero, we must modify the distance function, which leads to our next theorem.

\begin{theorem}\label{thm2}
    Let $r_t = \delta e^{-t|z|^2}$ for $t > 0$. The level sets defined by $r_t = c$, for constants $c < 0$ sufficiently close to $0$, are strongly pseudoconvex within an interior neighborhood of $\partial\Omega$.
\end{theorem}

\section{Preliminary}\label{pre}
Let $\nabla=\sum_i\frac{\partial}{\partial x^i}\otimes\frac{\partial}{\partial x^i}$ be the (real) gradient. If $f$ is a scalar function, then $\nabla f=\sum_i\frac{\partial}{\partial x^i}\otimes\frac{\partial f}{\partial x^i}$. Similarly, let $\D=\sum_i \frac{\partial}{\partial z^i}\otimes \frac{\partial}{\partial\bar z^i}$ be the complex gradient and $\D f=\sum_i \frac{\partial}{\partial z^i}\otimes \frac{\partial f}{\partial\bar z^i}$. Note that \[\D\rho+\overline{\D}\rho=\frac{1}{2}\nabla\rho.\] When $f=\rho$, we define the normalized complex gradient of $\rho$ to be $\N_\rho=|\D\rho|^{-1}\D\rho$. So, \[\N_\rho+\overline{\N}_\rho=\frac{1}{2}|\D\rho|^{-1}\nabla\rho.\] It is easy to see that $|\N_\rho|=1$. Also, $|\D\rho|^2=\frac{1}{2}\sum_i\left|\frac{\partial\rho}{\partial\bar z^i}\right|^2=\frac{1}{8}|\nabla\rho|^2$ and $\N_\rho\rho=|\D\rho|^{-1}\D\rho\cdot\rho=|\D\rho|^{-1}\sum_i\left|\frac{\partial\rho}{\partial\bar z^i}\right|^2=2|\D\rho|$. It is easy to check that $|\D\delta|^2=\frac{1}{8}$. Let $\N=\N_\delta$ and so $\N+\overline{\N}=\sqrt{2}\nabla\delta$.

To sum up, we have the following:
\begin{enumerate}
    \item $\N_\rho=\N$ on $\partial\Omega$
    \item $\N$ and $\N_\rho$ are both unit vectors in $\mathbb{C}^n$
    \item $\N+\overline{\N}=\sqrt{2}\nabla\delta$ in $\mathbb{C}^n$
    \item $\N_\rho\rho=\frac{1}{\sqrt{2}}|\nabla\rho|$ and $\N\delta=\frac{1}{\sqrt{2}}$ in $\mathbb{C}^n$
    \item $\inn{\nabla\delta, \frac{\partial}{\partial z^i}}=\frac{\partial\delta}{\partial\bar z^i}$
\end{enumerate}

In this note, $L \in \mathscr{C}^\infty(\mathbb{C}^n, T^{(1,0)}\mathbb{C}^n)$ denotes a vector field satisfying $L\delta = 0$ in a neighborhood of $\partial\Omega$. Additionally, we use $\inn{\cdot, \cdot}$ to denote the standard inner product on $\mathbb{C}^n$.

\section{The negativity brought by the distance function}

In this section, we prove Theorem \ref{thm1}.
\begin{proof}[The proof of Theorem \ref{thm1}]
    Since $\delta$ is the signed distance function, we have that $\nabla_{\nabla\delta}\nabla\delta=0$ because the integral curves of $\nabla\delta$ are unit speed geodesics (see \citet{Pe06}). This means $\nabla_{\N+\overline{\N}}(\N+\overline{\N})=0$ and consequently, $\inn{\nabla_{L}\nabla_{\N+\overline{\N}}(\N+\overline{\N}), L}=0$. Because $\overline{\N}\in\mathscr{C}^\infty(\mathbb{C}^n, T^{(0,1)}\mathbb{C}^n)$ and the covariant derivative $\nabla_{L}\nabla_{\N+\overline{\N}}\overline{\N}\in\mathscr{C}^\infty(\mathbb{C}^n, T^{(0,1)}\mathbb{C}^n)$, we can see that \[\inn{\nabla_{L}\nabla_{\N+\overline{\N}}\N, L}=\inn{\nabla_{L}\nabla_{\N+\overline{\N}}(\N+\overline{\N}), L}=0.\] This means \begin{equation}\label{1}
        \inn{\nabla_{L}\nabla_{\N}\N, L}=-\inn{\nabla_{L}\nabla_{\overline{\N}}\N, L}.
    \end{equation}

    We restrict to a coordinate chart $U$. Observe that, letting $\{L_i\}_{i=1}^n\subset \mathscr{C}^\infty(\mathbb{C}^n, T^{(1,0)}\mathbb{C}^n)$ be an orthonormal basis of $T^{(1,0)}\mathbb{C}^n$ everywhere in $U$ with $\N=L_n$, \[\begin{split}
         &\inn{\nabla_{L}\nabla_{\overline{\N}}\N, L}\\
        =&\inn{\nabla_{\overline{\N}}\nabla_L \N, L}+\inn{\nabla_{[L,\overline{\N}]}\N, L}\\
        =&\inn{\nabla_{\overline{\N}}\nabla_L \N, L}+\sum_{i=1}^{n}\inn{\nabla_{L_i}\N, L}\inn{[L,\overline{\N}], L_i}+\sum_{i=1}^{n}\inn{\nabla_{\overline{L}_i}\N, L}\inn{[L,\overline{\N}], \overline{L}_i}\\
        =&\inn{\nabla_{\overline{\N}}\nabla_L \N, L}-\sum_{i=1}^{n}\inn{\nabla_{L_i}\N, L}\inn{\nabla_{\overline{\N}}L, L_i}+\sum_{i=1}^{n}\inn{\nabla_{\overline{L}_i}\N, L}\inn{\nabla_{L}\overline{\N}, \overline{L}_i}\\
        =&\inn{\nabla_{\overline{\N}}\nabla_L \N, L}+\sum_{i=1}^{n}\inn{\nabla_{L_i}\N, L}\inn{L, \nabla_{\N} L_i}+\sum_{i=1}^{n}\inn{\N, \nabla_{L_i}L}\inn{\nabla_{L}L_i, \N}.
    \end{split}\]
    Note that $\inn{\nabla_L L_i, \N}=\inn{\nabla_{L_i} L, \N}+\inn{[L, L_i], \N}$, but \[\inn{[L, L_i], \N}=\sqrt{2}\inn{[L, L_i], \nabla\delta}=\sqrt{2}LL_i\delta-\sqrt{2}L_iL\delta=0\] for all $1\leq i\leq n$. So \[\inn{\nabla_{L}\nabla_{\overline{\N}}\N, L}
       =\inn{\nabla_{\overline{\N}}\nabla_L \N, L}+\sum_{i=1}^{n}\inn{\nabla_{L_i}\N, L}\inn{L, \nabla_{\N} L_i}+\sum_{i=1}^{n}|\inn{\N, \nabla_{L_i}L}|^2.\]

       When restricting to $\partial\Omega$, due to $0=\sqrt{2}\Hessian_\delta(L, L_i)=\inn{\nabla_L\N, L_i}$ for $1\leq i\leq n-1$, \[\begin{split}
           \left.\inn{\nabla_{L}\nabla_{\overline{\N}}\N, L}\right|_{\partial\Omega}
       &=\left.\inn{\nabla_{\overline{\N}}\nabla_L \N, L}\right|_{\partial\Omega}+\left.|\inn{L, \nabla_{\N} \N}|^2\right|_{\partial\Omega}+\left.\sum_{i=1}^{n}|\inn{\N, \nabla_{L_i}L}|^2\right|_{\partial\Omega}\\
       &=\overline{\N}\left.\inn{\nabla_L \N, L}\right|_{\partial\Omega}+\left.\sum_{i=1}^{n}|\inn{\N, \nabla_{L_i}L}|^2\right|_{\partial\Omega}.
       \end{split}\]

By (\ref{1}), we obtain that \begin{equation}\label{2}
    -\left.\inn{\nabla_{L}\nabla_{\N}\N, L}\right|_{\partial\Omega}
       =\overline{\N}\left.\inn{\nabla_L \N, L}\right|_{\partial\Omega}+\left.\sum_{i=1}^{n}|\inn{\N, \nabla_{L_i}L}|^2\right|_{\partial\Omega}.
\end{equation}

       Since on $\partial\Omega$, $\Hessian_\delta(L, L)=\sqrt{2}\inn{\nabla_L\N, L}$ reaches its minimum of $0$ at weakly pseudoconvex points, we obtain $(\N-\overline{\N})\inn{\nabla_L\N, L}=0$ at those weakly pseudoconvex points. Consequently, at weakly pseudoconvex points, $\N\inn{\nabla_L\N, L}=\overline{\N}\inn{\nabla_L\N, L}$.
On the other hand,
\[\begin{split}
    &\left.\inn{\nabla_L\nabla_{\N}\N, L}\right|_{\partial\Omega}=\left.\inn{\nabla_{\N}\nabla_L\N, L}\right|_{\partial\Omega}+\left.\inn{\nabla_{[L, \N]}\N, L}\right|_{\partial\Omega}\\
    =&\left.\N\inn{\nabla_L\N,L}\right|_{\partial\Omega}-\left.\inn{\nabla_L\N,\nabla_{\overline{\N}}L}\right|_{\partial\Omega}=\left.\N\inn{\nabla_L\N,L}\right|_{\partial\Omega}+\left.|\inn{\nabla_L\N,\N}|^2\right|_{\partial\Omega}.
\end{split}\]
   Consequently, by (\ref{1}), (\ref{2}) and the fact $\N\inn{\nabla_L\N, L}=\overline{\N}\inn{\nabla_L\N, L}$ on weakly pseudoconvex points, we obtain that \[\left.\inn{\nabla_L\nabla_{\N}\N, L}\right|_{\partial\Omega}=-\left.\inn{\nabla_{L}\nabla_{\N}\N, L}\right|_{\partial\Omega}-\left.\sum_{i=1}^{n}|\inn{\N, \nabla_{L_i}L}|^2\right|_{\partial\Omega}+\left.|\inn{\nabla_L\N,\N}|^2\right|_{\partial\Omega}.\]

   Since $\inn{\N, \nabla_{\N} L}=\inn{\N, \nabla_L \N}$, we obtain that \[2\left.\inn{\nabla_L\nabla_{\N}\N, L}\right|_{\partial\Omega}=-\left.\sum_{i=1}^{n-1}|\inn{\N, \nabla_{L_i}L}|^2\right|_{\partial\Omega}\leq 0.\]
\end{proof}
\begin{corollary}\label{cor1}
  Let $\Omega$ be a bounded pseudoconvex domain in $\mathbb{C}^n$ with smooth boundary. Then consider the domain $\Omega_{-\epsilon}:=\{x\in\Omega: \delta(x)\leq -\epsilon\}$ for $\epsilon>0$.  The $\Omega_{-\epsilon}$ makes sense for small $\epsilon$. 
    
    Let $x \in \partial\Omega_{-\epsilon}$ and $p \in \partial\Omega$ be such that $|x - p| = \distance(x, \partial\Omega)$. The rate of change $\left.(\N + \overline{\N})\inn{\nabla_L\N, L}\right|_{\partial\Omega}$ is nonpositive at any point $p \in \partial\Omega$ satisfying $\left.\Hessian_{\delta}(L, L)\right|_p = 0$. If $\left.\Hessian_{\delta}(L, L)\right|_{x\in\partial\Omega_{-\epsilon}}=0$ for all small $\epsilon$, then for any $\epsilon'\in [0,\epsilon]$, letting $p'=[x,p]\cap\partial\Omega_{-\epsilon'}$, we have $\left.\inn{\nabla_L\N,\N}\right|_{p'\in\partial\Omega_{-\epsilon'}}= 0$, where $[x,p]$ denotes the line segment connecting $x$ with $p$.   
\end{corollary}

\begin{proof}
    Since \[\begin{split}
    &0\geq\left.\inn{\nabla_L\nabla_{\N}\N, L}\right|_{\partial\Omega}=\left.\inn{\nabla_{\N}\nabla_L\N, L}\right|_{\partial\Omega}+\left.\inn{\nabla_{[L, \N]}\N, L}\right|_{\partial\Omega}\\
    =&\left.\N\inn{\nabla_L\N,L}\right|_{\partial\Omega}-\left.\inn{\nabla_L\N,\nabla_{\overline{\N}}L}\right|_{\partial\Omega}=\left.\N\inn{\nabla_L\N,L}\right|_{\partial\Omega}+\left.|\inn{\nabla_L\N,\N}|^2\right|_{\partial\Omega},
\end{split}\]
it implies, \[\left.\N\inn{\nabla_L\N,L}\right|_{\partial\Omega}+\left.|\inn{\nabla_L\N,\N}|^2\right|_{\partial\Omega}\leq 0.\] In particular, $\left.(\N+\overline{\N})\inn{\nabla_L\N,L}\right|_{\partial\Omega}\leq 0$ and if $\left.(\N+\overline{\N})\inn{\nabla_L\N,L}\right|_{\partial\Omega}= 0$ then $\left.\inn{\nabla_L\N,\N}\right|_{\partial\Omega}= 0$.

\end{proof}

\section{Proof of Theorem \ref{thm2}}
In this section, we prove Theorem \ref{thm2}.
\subsection{Perturbation of $\delta$}

We replace $\delta$ by $r_t=\delta e^{-t|z|^2}$. By computation, we obtain that \[\begin{split}
     &\Hessian_{r_t}(L,L)\\
    =&L\overline{L}r_t-\nabla_L\overline{L}r_t\\
    =&L\overline{L}\delta e^{-t|z|^2}-\nabla_L\overline{L}\delta e^{-t|z|^2}\\
    =&\delta L\overline{L}e^{-t|z|^2}-\delta\nabla_L\overline{L} e^{-t|z|^2}- e^{-t|z|^2}\nabla_L\overline{L}\delta\\
    =&e^{-t|z|^2}\Hessian_\delta(L,L)+\delta\Hessian_{e^{-t|z|^2}}(L,L).
\end{split}
\]

We compute \[\begin{split}
     &\Hessian_{e^{-t|z|^2}}(L,L)\\
    =&L\overline{L}e^{-t|z|^2}-\nabla_{L}\overline{L}e^{-t|z|^2}\\
    =&-tL e^{-t|z|^2}\overline{L}|z|^2+te^{-t|z|^2}\nabla_{L}\overline{L}|z|^2\\
    =&-t e^{-t|z|^2}L\overline{L}|z|^2+t^2e^{-t|z|^2}L |z|^2\overline{L}|z|^2+te^{-t|z|^2}\nabla_{L}\overline{L}|z|^2\\
    =&-te^{-t|z|^2}\Hessian_{|z|^2}(L, L)+t^2e^{-t|z|^2}\left|L|z|^2\right|^2
\end{split}\]

On the other hand, let $L_t$ be tangential to the level set $\{z\in\Omega: \delta e^{t|z|^2}=c\}$. Observe that $\nabla r_t=e^{-t|z|^2}\nabla\delta-t\delta e^{-t|z|^2}\nabla|z|^2$. Compute \[\begin{split}
     &\inn{\nabla_{L_t}\nabla r_t, \nabla_{L_t}}\\
    =&\inn{\nabla_{L_t} e^{-t|z|^2}\nabla\delta, \nabla_{L_t}}-t\inn{\nabla_{L_t}\delta e^{-t|z|^2}\nabla|z|^2, \nabla_{L_t}}.
\end{split}\]

We now compute, letting $L^t_i$ be the orthonormal frame compatible with the level set $\{z\in\Omega: \delta e^{-t|z|^2}=c\}$, \[\begin{split}
     &\N_t\inn{\nabla_{L_t}\nabla r_t, L_t}\\
    =&\inn{\nabla_{\N_t}\nabla_{L_t}\nabla r_t, L_t}+\inn{\nabla_{L_t}\nabla r_t, \nabla_{\overline{\N}_t}L_t}\\
    =&\inn{\nabla_{L_t}\nabla_{\N_t}\nabla r_t, L_t}+\inn{\nabla_{[\N_t,L_t]}\nabla r_t, L_t}+\inn{\nabla_{L_t}\nabla r_t, \nabla_{\overline{\N}_t}L_t}\\
    =&\inn{\nabla_{L_t}\nabla_{\N_t}\nabla r_t, L_t}+\sum_{i=1}^n\inn{[\N_t,L_t], L^t_i}\inn{\nabla_{L^t_i}\nabla r_t, L_t}+\inn{\nabla_{L_t}\nabla r_t, \nabla_{\overline{\N}_t}L_t}.
\end{split}\]

Restricting it to the boundary $\partial\Omega$, note that $L_t=L$ and $\N_t=\N$ on $\partial\Omega$. If $\Hessian_\delta(L,L)=0$, we get on $\partial\Omega$, \[\N_t\inn{\nabla_{L_t}\nabla r_t, L_t}=\inn{\nabla_{L_t}\nabla_{\N_t}\nabla r_t, L_t}+\inn{[\N_t,L_t], \N_t}\inn{\nabla_{\N_t}\nabla r_t, L_t}-\inn{\nabla_{L_t}\nabla r_t, \N_t}\inn{\nabla_{\N_t}\N_t,L_t}=I+II+III.\]

We will compute $\N_t$. Based on the definition $\N_t=\frac{\D r_t}{|\D r_t|}$. Since \[\begin{split}
     &\D r_t=\sum_i\frac{\partial}{\partial z^i}\frac{\partial r_t}{\partial\bar{z}^i}=\sum_i\frac{\partial}{\partial z^i}\otimes e^{-t|z|^2}\left(\frac{\partial\delta}{\partial\bar{z}^i}-t\delta z^i\right)\\
    =&e^{-t|z|^2}\sum_i\frac{\partial}{\partial z^i}\otimes \frac{\partial\delta}{\partial\bar{z}^i}-t\delta e^{-t|z|^2}\sum_i\frac{\partial}{\partial z^i}\otimes  z^i=e^{-t|z|^2}\D\delta-t\delta e^{-t|z|^2}\sum_i\frac{\partial}{\partial z^i}\otimes  z^i.
\end{split}\]

We compute that \[|\D r_t|=e^{-t|z|^2}\sqrt{\frac{1}{2}\sum_i \left|\frac{\partial\delta}{\partial\bar{z}^i}-t\delta z^i\right|^2}\] and \[\N_t=\sqrt{2}\sum_i\frac{\partial}{\partial z^i}\otimes\frac{\frac{\partial\delta}{\partial\bar{z}^i}-t\delta z^i}{\sqrt{\sum_i \left|\frac{\partial\delta}{\partial\bar{z}^i}-t\delta z^i\right|^2}}.\]

We observe that \[\begin{split}
     &\inn{\nabla_{\N_t}\N_t,L_t}=\inn{\nabla_{\N_t}\frac{\D r_t}{|\D r_t|},L_t}=\frac{1}{|\D r_t|}\inn{\nabla_{\N_t}\D r_t,L_t}+\left(\N_t\frac{1}{|\D r_t|}\right)\inn{\D r_t, L_t}\\
    =&\frac{1}{2|\D r_t|}\inn{\nabla_{\N_t}\nabla r_t,L_t}+\left(\N_t\frac{1}{|\D r_t|}\right)\inn{\D r_t, L_t}=\frac{1}{2|\D r_t|}\inn{\nabla_{\N_t}\nabla r_t,L_t}.
\end{split}\]

So, \[III=-\inn{\nabla_{L_t}\nabla r_t, \N_t}\inn{\nabla_{\N_t}\N_t,L_t}=-\frac{1}{2|\D r_t|}\inn{\nabla_{\N_t}\nabla r_t,L_t}\inn{\nabla_{L_t}\nabla r_t, \N_t}=-\frac{1}{2|\D r_t|}|\inn{\nabla_{\N_t}\nabla r_t,L_t}|^2.\]

Also, \[\begin{split}
     &\inn{[\N_t,L_t], \N_t}
    =\inn{[\N_t,L_t], \frac{\D r_t}{|\D r_t|}}
    =\frac{1}{|\D r_t|}\inn{[\N_t,L_t],\D r_t+\overline{\D} r_t}
    = \frac{1}{2|\D r_t|}\inn{[\N_t,L_t],\nabla r_t}\\
    =& \frac{1}{2|\D r_t|}[\N_t,L_t] r_t
    = \frac{1}{2|\D r_t|}(\N_tL_t r_t-L_t\N_t r_t)
    = -\frac{2}{2|\D r_t|}L_t|\D r_t|=-\frac{1}{2|\D r_t|^2}L_t|\D  r_t|^2.
\end{split}\]

We compute that \[\begin{split}
     &L_t|\D  r_t|^2\\
    =&\frac{1}{2}L_t e^{-2t|z|^2}\sum_i  \left|\frac{\partial\delta}{\partial\bar{z}^i}-t\delta z^i\right|^2\\
    =&-t e^{-2t|z|^2}\sum_i  \left|\frac{\partial\delta}{\partial\bar{z}^i}-t\delta z^i\right|^2L_t |z|^2+\frac{1}{2} e^{-2t|z|^2}\sum_i L_t \left|\frac{\partial\delta}{\partial\bar{z}^i}-t\delta z^i\right|^2.
\end{split}\]

Restricting to $\partial\Omega$, we obtain that on $\partial\Omega$, \[\inn{[\N_t,L_t], \N_t}=t L_t |z|^2.\]

So, on $\partial\Omega$, \[II=\inn{[\N_t,L_t], \N_t}\inn{\nabla_{\N_t}\nabla r_t, L_t}=t (L_t |z|^2)\inn{\nabla_{\N_t}\nabla r_t, L_t}.\]
\subsection{The calculation of $\inn{\nabla_{L}\nabla_{\N}\nabla r_t, L}$}
We observe that \[\begin{split}
     &\nabla_{L}\N_t=\sqrt{2} \sum_i\frac{\partial}{\partial z^i}\otimes L\left(\frac{\frac{\partial\delta}{\partial\bar{z}^i}-t\delta z^i}{\sqrt{\sum_i \left|\frac{\partial\delta}{\partial\bar{z}^i}-t\delta z^i\right|^2}}\right)\\
    =&\sqrt{2} \sum_i\frac{\partial}{\partial z^i}\otimes \frac{\sqrt{\sum_i \left|\frac{\partial\delta}{\partial\bar{z}^i}-t\delta z^i\right|^2}L\left(\frac{\partial\delta}{\partial\bar{z}^i}-t\delta z^i\right)-\left(\frac{\partial\delta}{\partial\bar{z}^i}-t\delta z^i\right)L\sqrt{\sum_i \left|\frac{\partial\delta}{\partial\bar{z}^i}-t\delta z^i\right|^2}}{\sum_i \left|\frac{\partial\delta}{\partial\bar{z}^i}-t\delta z^i\right|^2}\\
    =&\sqrt{2} \sum_i\frac{\partial}{\partial z^i}\otimes \frac{\sqrt{\sum_i \left|\frac{\partial\delta}{\partial\bar{z}^i}-t\delta z^i\right|^2}L\left(\frac{\partial\delta}{\partial\bar{z}^i}-t\delta z^i\right)-\left(\frac{\partial\delta}{\partial\bar{z}^i}-t\delta z^i\right)L\sqrt{\sum_i \left|\frac{\partial\delta}{\partial\bar{z}^i}\right|^2-\sum_i\left|t\delta z^i\right|^2-2\sum_i\Re\frac{\partial\delta}{\partial\bar{z}^i}t\delta \bar z^i}}{\sum_i \left|\frac{\partial\delta}{\partial\bar{z}^i}-t\delta z^i\right|^2}.
\end{split}\]

Restricting to $\partial\Omega$, both $\delta=0$ and $L\delta=0$ on $\partial\Omega$. Consequently, $\left.\nabla_L\N_t\right|_{\partial\Omega}=\left.\nabla_L\N\right|_{\partial\Omega}$.

Because of this, restricting \[\nabla_{L}\nabla_{\N_t}\nabla r_t=(\nabla^2\nabla r_t) (L, \N_t)+\nabla_{\nabla_L\N_t}\nabla r_t\] on $\partial\Omega$, we obtain that on $\partial\Omega$, \[\nabla_{L}\nabla_{\N_t}\nabla r_t=(\nabla^2\nabla r_t) (L, \N)+\nabla_{\nabla_L\N}\nabla r_t=\nabla_{L}\nabla_{\N}\nabla r_t.\]

We calculate that \[\begin{split}
     &\inn{\nabla_{L}\nabla_{\N}\nabla r_t, L}\\
    =&-t\inn{\nabla_{L}\nabla_{\N}\delta e^{-t|z|^2}\nabla |z|^2 , L}+\inn{\nabla_{L}\nabla_{\N} e^{-t|z|^2}\nabla\delta  , L}\\
    =&-t\inn{\nabla_{L}(\N\delta) e^{-t|z|^2}\nabla |z|^2 , L}-t\inn{\nabla_{L}\delta\nabla_{\N} e^{-t|z|^2}\nabla |z|^2 , L}+\inn{\nabla_{L}(\N e^{-t|z|^2})\nabla\delta  , L}+\inn{\nabla_{L} e^{-t|z|^2}\nabla_{\N}\nabla\delta  , L}\\
    =&-t(\N\delta)\inn{ \nabla_{L}e^{-t|z|^2}\nabla |z|^2 , L}-t\inn{\nabla_{L}\delta\nabla_{\N} e^{-t|z|^2}\nabla |z|^2 , L}+(\N e^{-t|z|^2})\inn{\nabla_{L}\nabla\delta  , L}+\inn{\nabla_{L} e^{-t|z|^2}\nabla_{\N}\nabla\delta  , L}.
\end{split}\]

Restricting to $\partial\Omega$ and assuming $\Hessian_{r_t}(L_t, L_t)=0$ on $\partial\Omega$, \[
     I=\inn{\nabla_{L}\nabla_{\N}\nabla r_t, L}
    =-\frac{t}{\sqrt{2}}e^{-t|z|^2}+\frac{t^2}{\sqrt{2}}e^{-t|z|^2}\left|L|z|^2\right|^2+e^{-t|z|^2}\inn{\nabla_{L} \nabla_{\N}\nabla\delta  , L}-te^{-t|z|^2}(L|z|^2)\inn{\nabla_{\N}\nabla\delta  , L}.\]

Considering the difference between $\inn{\nabla_{\N}\nabla\delta  , L}$ and $\inn{\nabla_{\N_t}\nabla r_t  , L_t}$ on $\partial\Omega$, \[\begin{split}
     &\inn{\nabla_{\N}\nabla r_t  , L}=\inn{\nabla_{\N}\nabla \delta e^{-t|z|^2}  , L}=-t\inn{\nabla_{\N}\delta e^{-t|z|^2}\nabla  |z|^2  , L}+\inn{\nabla_{\N} e^{-t|z|^2}\nabla \delta  , L}\\
    =&-\frac{t}{\sqrt{2}}e^{-t|z|^2}(\overline{L}|z|^2)-t\delta\inn{\nabla_{\N} e^{-t|z|^2}\nabla  |z|^2  , L}+e^{-t|z|^2}\inn{\nabla_{\N} \nabla \delta  , L}.
\end{split}\]

Restricting to $\partial\Omega$, \[\inn{\nabla_{\N_t}\nabla r_t  , L_t}=-\frac{t}{\sqrt{2}}e^{-t|z|^2}(\overline{L}|z|^2)+e^{-t|z|^2}\inn{\nabla_{\N} \nabla \delta  , L}.\]

So, restricting to $\partial\Omega$ and assuming $\Hessian_{r_t}(L_t, L_t)=0$ on $\partial\Omega$, \[\begin{split}
     &I=\inn{\nabla_{L}\nabla_{\N}\nabla r_t, L}\\
    =&-\frac{t}{\sqrt{2}}e^{-t|z|^2}+\frac{t^2}{\sqrt{2}}e^{-t|z|^2}\left|L|z|^2\right|^2+e^{-t|z|^2}\inn{\nabla_{L} \nabla_{\N}\nabla\delta  , L}-t(L|z|^2)\left(\inn{\nabla_{\N_t}\nabla r_t  , L_t}+\frac{t}{\sqrt{2}}e^{-t|z|^2}(\overline{L}|z|^2)\right)\\
    =&-\frac{t}{\sqrt{2}}e^{-t|z|^2}+e^{-t|z|^2}\inn{\nabla_{L} \nabla_{\N}\nabla\delta  , L}-t(L|z|^2)\inn{\nabla_{\N_t}\nabla r_t  , L_t}.
\end{split}
     \]

\subsection{Total}
So, in total, on $\partial\Omega$ assuming $\Hessian_{r_t}(L_t, L_t)=0$, \[\begin{split}
     &\N_t\inn{\nabla_{L_t}\nabla r_t, L_t}\\
    =&-\frac{t}{\sqrt{2}}e^{-t|z|^2}+e^{-t|z|^2}\inn{\nabla_{L} \nabla_{\N}\nabla\delta  , L}-t(L|z|^2)\inn{\nabla_{\N_t}\nabla r_t  , L_t}+t (L_t |z|^2)\inn{\nabla_{\N_t}\nabla r_t, L_t}-\frac{1}{2|\D r_t|}|\inn{\nabla_{\N_t}\nabla r_t,L_t}|^2\\
    =&-\frac{t}{\sqrt{2}}e^{-t|z|^2}+e^{-t|z|^2}\inn{\nabla_{L} \nabla_{\N}\nabla\delta  , L}-\frac{1}{2|\D r_t|}|\inn{\nabla_{\N_t}\nabla r_t,L_t}|^2\\
    \leq&-\frac{t}{\sqrt{2}}e^{-t|z|^2}<0.
\end{split}\]

For any $t>0$, the $\N_t\inn{\nabla_{L_t}\nabla r_t, L_t}<\frac{-t}{2\sqrt{2}}e^{-t|z|^2}$ in a neighborhood of $\Sigma_{L_t}$ by  continuity. This implies all level sets of $r_t$ are strongly pseudoconvex. Indeed, in the neighborhood $U_{L_t}$ of $\Sigma_{L_t}$, as $\epsilon$ decays, the Levi form increases and becomes positive. For the $\partial\Omega\backslash U_{L_t}$, the Levi-form is bounded below by a positive number, say $c$. It is obvious that an interior tubular neighborhood of $\partial\Omega\backslash U_{L_t}$ will have strictly positive Levi form. By the compactness of $L_t/|L_t|$, we prove that $\Omega_{-\epsilon}$ defined by the level sets of $r_t$ is strongly pseudoconvex domain for small $\epsilon$.

This completes the proof of Theorem \ref{thm2}.

\printbibliography
\end{document}